\documentclass[a4paper,12pt]{amsart}

\usepackage{amssymb}
\usepackage{mhequ}
\usepackage{scalerel}
\renewcommand{\div}{\mbox{div}}

\usepackage{cite}

\DeclareFontFamily{OT1}{rsfs}{}
\DeclareFontShape{OT1}{rsfs}{m}{n}{ <-7> rsfs5 <7-10> rsfs7 <10-> rsfs10}{}
\DeclareMathAlphabet{\mathscr}{OT1}{rsfs}{m}{n}

\newcommand{\eq}[1]{\eqref{#1}}

\newcommand{\bel}[1]{\begin{equation}\label{#1}}
\newcommand{\beal}[1]{\begin{eqnarray}\label{#1}}
\newcommand{\beadl}[1]{\begin{deqarr}\label{#1}}
\newcommand{\eeadl}[1]{\arrlabel{#1}\end{deqarr}}
\newcommand{\eeal}[1]{\label{#1}\end{eqnarray}}
\newcommand{\eead}[1]{\end{deqarr}}
\newcommand{\eea}{\end{eqnarray}}
\newcommand{\eeaa}{\end{eqnarray*}}

\newcommand{\be}{\begin{equation}}
\newcommand{\ee}{\end{equation}}

\DeclareFontFamily{OT1}{rsfs}{}
\DeclareFontShape{OT1}{rsfs}{m}{n}{ <-7> rsfs5 <7-10> rsfs7 <10->
rsfs10}{} \DeclareMathAlphabet{\mycal}{OT1}{rsfs}{m}{n}

\newcounter{mnotecount}[section]

\newcommand{\N}{{\Bbb N}}

\newcommand{\rmnote}[1]{}

\newcommand{\Ric}{\operatorname{Ric}}

\def\mysavedown#1{\edef\mysubs{\mysubs#1}}
\def\mysaveup#1{\edef\mysups{\mysups#1}}
\def\mydown#1{{\mytensor}_{\vphantom{\mysubs}#1}}
\def\myup#1{{\mytensor}^{\vphantom{\mysups}#1}}
\def\tensor#1#2{
  #1
  \def\mytensor{\vphantom{#1}}
  \def\mysubs{\relax}
  \def\mysups{\relax}
  \let\down=\mysavedown
  \let\up=\mysaveup
  #2
  \let\down=\mydown
  \let\up=\myup
  #2
  }

\newcommand{\Riem}{\operatorname{Riem}}
\newcommand{\Hess}{\operatorname{Hess}}

\newcommand{\Tr}{\operatorname{Tr}}

\newcommand{\R}{\mathbb R}

\renewcommand{\div}{\operatorname{div}}

\DeclareMathOperator{\Vol}{Vol}

\renewcommand{\phi}{\varphi}
\renewcommand{\epsilon}{\varepsilon}

\def\crn#1#2{{\vcenter{\vbox{
        \hbox{\kern#2pt \vrule width.#2pt height#1pt
           }
          \hrule height.#2pt}}}}

\newcommand{\Ein}{\operatorname{Ein}}

\renewcommand{\hbar}{{\overline h}}

\newcommand{\pre}[2]{{{\vphantom{#2}}^{#1}}\kern-.2ex{#2}}

\theoremstyle{plain}
\newtheorem{theorem}{Theorem}[section]

\newtheorem{proposition}[theorem]{Proposition}
\newtheorem{corollary}[theorem]{Corollary}

\theoremstyle{definition}

\newtheorem{definition}[theorem]{Definition}

\newtheorem{remark}[theorem]{Remark}

\numberwithin{equation}{section}

\begin{document}
\title[ Ricci curvature near Einstein manifolds with boundary] 
{Prescribed Ricci curvature near an Einstein manifold with boundary}

\author[E. Delay]{Erwann Delay}
\address{Laboratoire de Math\'ematiques d'Avignon,
 Fac. des Sciences,
 F84916 Avignon, France\newline
$\mbox{ }\;\;$F.R.U.M.A.M. – CNRS Aix Marseille Université– F13 331 Marseille, France}
\date{March 24, 2025}
\email{Erwann.Delay@univ-avignon.fr}
\urladdr{https://erwanndelay.wordpress.com/}

\begin{abstract}
Let $(M,g)$ be a compact Einstein Riemannian manifold with boundary. 
We show that under certain conditions, the map that associates to a metric on $M$ its Ricci curvature, its induced conformal class on the boundary, and its mean curvature on the boundary is locally invertible near $g$.
The contravariant Ricci operator, as well as other operators such as the Einstein operator, are also studied.
\end{abstract}

\maketitle

\\
\newline

\tableofcontents

\section{Introduction}\label{section:intro}
On a Riemannian manifold $(M,g)$, we denote by $\Ric(g)$ its Ricci curvature.
We first consider the (field of) geometric symmetric 2-tensors of the form
$$
\Ric_\Lambda(g):=\Ric(g)+\Lambda g,
$$
where $\Lambda$ is a constant.
This tensor is geometrically natural in the sense that for any sufficiently regular diffeomorphism $\varphi$,
$$
\varphi^*\Ric_\Lambda(g)=\Ric_\Lambda(\varphi^*g).
$$
Here, we address the problem of inverting the operator $\Ric_\Lambda$.
Given a symmetric tensor field $R_\Lambda$ on $M$, we seek a Riemannian metric $g$ such that
\bel{mainequation}
\Ric_\Lambda(g)=R_\Lambda.
\ee
This requires solving a particularly complex quasi-linear system.

The prescribed Ricci curvature problem dates back to the 1980s.
DeTurck \cite{Deturck:ricci}, in 1981, first proved a local existence result near a point $p$ in $\mathbb{R}^n$ under the intrinsic assumption that the matrix of $R(p)$ is invertible (see also \cite{Deturckrank1}).

Subsequently, global results were obtained: Hamilton \cite{Hamilton1984}, in 1984, treated the case of the unit sphere in $\mathbb{R}^{n+1}$ (with $n>2$), proving a local inversion result near the standard metric.

Such local inversion techniques were later adapted to certain Einstein manifolds \cite{Delay:etude}, \cite{DelayHerzlich}, \cite{DeturckEinstein}, \cite{Delay:study}, \cite{Delanoe2003}, \cite{Delay:ricciAE}, and then to parallel Ricci manifolds \cite{Delay:RicciPara}, \cite{Delay:SSR}. The problem of prescribed Ricci curvature has also been studied on Lie groups and homogeneous spaces with completely different methods (see \cite{Gaskins} and references therein).

Obstruction results to the inversion of Ricci curvature also exist \cite{Deturck-Koiso}, \cite{Baldes1986}, \cite{Hamilton1984}, \cite{Delanoe1991}, \cite{Delay:etude}.

So far, the only results concerning manifolds with boundary is local near a boundary point \cite{Pulemotov:RicciBord} or for special domains of some cohomogeneity one
manifolds \cite{Pulemotov:Dirichlet}. Recent works on Einstein metrics on manifolds with boundary, such as \cite{ANandHUANG2024}, have inspired us to obtain {\it global} results on $M$ near an Einstein metric.

To illustrate our results simply, we state a particular case.
The induced boundary metric is denoted by $g^T$, and the mean curvature of the boundary by $H(g)$.

\begin{theorem}\label{maintheorem} 
Let $(M,g)$ be a smooth compact Einstein Riemannian manifold with boundary, with $\Ric(g)=\lambda g$. Suppose that $\lambda+\Lambda\neq0$ and that $-2\Lambda$ is not in the spectrum of the Hodge Laplacian acting on 1-forms with Dirichlet condition, nor in the spectrum of the Lichnerowicz Laplacian with ADN conditions.
Let $k\in\mathbb{N}\backslash\{0\}$ and $\alpha\in(0,1)$, 
then for every $r\in C^{k+2,\alpha}(M,\mathcal{S}_2)$ close to zero, every conformal class $[\gamma]$ close to $[g^T]$ in $ [C^{k+2,\alpha}(\partial M,S_2^+)]$, and every function $\mathcal{H}$ close to $H(g)$ in $C^{k+1,\alpha}(\partial M)$, there exists a unique $h$ close to zero in $C^{k+2,\alpha}(M,\mathcal{S}_2)$ such that
$$
\left\{
\begin{array}{ll}
\Ric_\Lambda(g+h)=\Ric_\Lambda(g)+r & \text{ in } M ,\\
 \mbox{}[(g+h)^T]=[\gamma] & \text{ on } \partial M,\\
H(g+h)=\mathcal{H} & \text{ on } \partial M.
\end{array}
\right.
$$
Moreover, the map $(r,[\gamma],\mathcal{H})\mapsto h$ is smooth from a neighborhood of $(0,[g^T], H(g))$ to a neighborhood of zero in the corresponding Banach spaces. 
\end{theorem}

The ADN condition, introduced by M. Anderson \cite{MAEinsteinB2008}, is a mixed Dirichlet-Neumann condition specified in Definition \ref{ConditionDN}.

Our solution's regularity is optimal, as can be seen by applying a low-regularity diffeomorphism to the equation (\ref{mainequation}).

We also establish an analogous result for the contravariant Ricci operator, which is better suited to negative curvature cases.

Furthermore, we extend this approach to other operators, such as the Einstein tensor.

Finally, we show that the image of certain Riemann-Christoffel-type operators forms smooth submanifolds in $C^{\infty}$.

\medskip

{\small\sc Acknowledgments}: I thank Lan-Hsuan Huang for insightful discussions concerning \cite{ANandHUANG2024} and \cite{MAEinsteinB2008}. This article was partially funded by the French National Research Agency (ANR) grants ANR-23-CE40-0010-02 and ANR-24-CE40-0702.

\section{Definitions, Notations, and Conventions}\label{sec:def}

We denote by $\nabla$ the Levi-Civita connection of $g$, by $\Ric(g)$ its Ricci curvature, and by $\Riem(g)$ its sectional Riemannian curvature.

Let ${\mathcal T}_p^q$ be the set of covariant tensors of rank $p$ and contravariant tensors of rank $q$.
When $p=2$ and $q=0$, we denote by ${\mathcal S}_2$ the subset of symmetric tensors,
which decomposes as ${\mathcal G} \oplus {\mathring{\mathcal S}_2}$, where ${\mathcal G}$ is the set of $g$-conformal tensors and 
${\mathring{\mathcal S}_2}$ is the set of traceless tensors (relative to $g$).
The set ${\mathcal S}_2^+$ denotes the subset of ${\mathcal S}_2$ consisting of positive definite tensors. 

For $k\in\N$, $\alpha\in(0,1)$ and a tensor bundle $E$,  $C^{k,\alpha}(M,E)$ is the usual Hölder space of tensor fields with $C^{k,\alpha}$ regularity.
The space of conformal class of metric of $C^{k,\alpha}$ regularity on the boundary is denoted by
$$
\begin{array}{lll}
[C^{k,\alpha}(\partial M,{\mathcal S}_2^+)]&=&\{[\gamma],\;\gamma\in C^{k,\alpha}(\partial M,{\mathcal S}_2^+)\}\\
&\cong&\{\gamma\in C^{k,\alpha}(\partial M,{\mathcal S}_2^+), |\gamma|=1\},
\end{array}
$$
where $|.|$ is the determinant (with respect to a fixed background metric, say $g^T$), and where for $\gamma\in C^{k,\alpha}(\partial M,{\mathcal S}_2^+)$, we identify its conformal class $[\gamma]$ with its unique representative of determinant 1.

We use Einstein's summation convention (indices range from $1$ to $n$), and we use 
$g_{ij}$ and its inverse $g^{ij}$ to raise or lower indices.

The Laplacian is defined by
$$
\triangle=-tr\nabla^2=\nabla^*\nabla,
$$
where $\nabla^*$ is the formal $L^2$ adjoint of $\nabla$. The Lichnerowicz Laplacian acting on symmetric 2-covariant tensor fields is \cite{Lichnerowicz:prop}
\bel{laplichne}
\triangle_L=\triangle+2(\Ric-\Riem),
\ee
where
$$(\Ric\; u)_{ij}=\frac{1}{2}[\Ric(g)_{ik}u^k_j+\Ric(g)_{jk}u^k_i],$$
and
$$
(\Riem \; u)_{ij}=\Riem(g)_{ikjl}u^{kl}.
$$
For $u$ a symmetric 2-covariant tensor, its divergence is defined by
$$ (\mbox{div}u)_i=-\nabla^ju_{ji}.$$ For a 1-form
$\omega$ on $M$, its divergence is defined by:
$$
d^*\omega=-\nabla^i\omega_i,
$$
and the symmetric part of its covariant derivatives:
$$
({\mathcal
L}\omega)_{ij}=\frac{1}{2}(\nabla_i\omega_j+\nabla_j\omega_i),$$
(note that ${\mathcal L}^*=\mbox{div}$).
The Hodge-de Rham Laplacian acting on 1-forms is denoted by
$$
\Delta_H=dd^*+d^*d=\Delta+\Ric.
$$

The Bianchi operator for symmetric 2-tensors into 1-forms is defined as:
$$
B_g(h)=\div_gh+\frac{1}{2}d(\Tr_gh).
$$

\section{Fredholm Properties and Isomorphisms}
We recall here the Fredholm properties of Laplacian-type operators plus zero-order terms acting on 1-forms or on symmetric two-tensors.
We begin by stating a well-known fact (see eg. \cite{MMMT}).

\begin{proposition}\label{DeltaHFredholm}
Let $k\in\N$, $\alpha\in(0,1)$, and $c$ be a real number. 
The operator from $C^{k+2,\alpha}(M,\mathcal T_1)$ to $C^{k,\alpha}(M,\mathcal T_1)\times C^{k+2,\alpha}(M,\mathcal T_1)_{|\partial M},$ given by
$$
\left\{
\begin{array}{ll}
(\Delta_H+c)\omega &\mbox{ in } M\\
\omega &\mbox{ on } \partial M\\
\end{array}
\right.
$$
is Fredholm of index 0.
\end{proposition}

Let $C_0^{k,\alpha}(M,\mathcal T_1)$ denote the space of $C^{k,\alpha}(M,\mathcal T_1)$ 1-forms satisfying the Dirichlet condition (i.e., vanishing on the boundary of $M$).

\begin{corollary}\label{DeltaHiso}
Let $k\in\N$, $\alpha\in(0,1)$, and $c$ be a real number. If $-c$ is not in the spectrum of the Hodge Laplacian $\Delta_H$ with Dirichlet boundary condition, then
$$
\Delta_H+c : C_0^{k+2,\alpha}(M,\mathcal T_1)\longrightarrow C^{k,\alpha}(M,\mathcal T_1),
$$
is an isomorphism.
\end{corollary}

We recall that the kernel of $\Delta_H+c$ with Dirichlet boundary condition is generically trivial for boundary deformations towards the interior (see, for example, \cite{ANandHUANG2024}, Lemma 1.2).\\

The following proposition involves the linearization $D H(g)$ of the mean curvature operator $H$ at $g$.
Even though we will not need the explicit formula, we recall one (see \cite{Lott:H}, for example, with a different sign convention).
Let $\nu$ be the outward-pointing normal to the boundary and
$ \mathbb I_{AB}=-\langle\nabla_{\partial_A}\nu,\partial_B\rangle$ be the second fundamental form of the boundary, so that $H=\Tr_{g^T}\mathbb I.$
Then we have
$$
DH(g)h=div_{g^T}(h(\nu,.)^T)+\frac12\Tr_{g^T}((\nabla_\nu h)^T)-\frac12 h(\nu,\nu)H.
$$

We can now state the following

\begin{proposition}\label{LcFredholm}
Let $k\in\N$, $\alpha\in(0,1)$, and $c$ be a real number. 
The operator  {\footnotesize
$$L_c: C^{k+2,\alpha}(M,\mathcal S_2)\longrightarrow C^{k,\alpha}(M,\mathcal S_2)\times C^{k+1,\alpha}(\partial M,\mathcal S_2)\times C^{k+2,\alpha}(\partial M,\mathring{\mathcal S}_2)\times C^{k+1,\alpha}(\partial M),$$}
given by
$$
L_c(h)=\left\{
\begin{array}{ll}
(\Delta_L+c)h &\mbox{in } M\\
B_g(h) &\mbox{on } \partial M\\
h^T-\frac{1}{n-1}\Tr_{g^T} (h^T) g^T&\mbox{on } \partial M\\
D H(g)h&\mbox{on } \partial M\\
\end{array}
\right.
$$
is Fredholm of index 0.
\end{proposition}

\begin{proof}
In \cite{MAEinsteinB2008}, M. Anderson proves that $L_c$ is elliptic and Fredholm of index zero when $c=-2\lambda$; the case of arbitrary $c$ follows immediately by homotopy.
\end{proof}

This proposition motivates the definition of a mixed condition between Dirichlet and Neumann, introduced by M. Anderson.

\begin{definition}[ADN condition]\label{ConditionDN}
We say that $h\in C^{k+2,\alpha}(M,\mathcal S_2)$ satisfies the ADN condition if
$$
\left\{
\begin{array}{ll}
B_g(h)=0 &\mbox{on } \partial M\\
h^T-\frac{1}{n-1}\Tr_{g^T} (h^T) g^T=0&\mbox{on } \partial M\\
D H (g)h=0&\mbox{on } \partial M\\
\end{array}
\right.
$$
\end{definition}

We thus immediately obtain

\begin{corollary}\label{DeltaLiso}
If $-c$ is not in the spectrum of the Lichnerowicz Laplacian with ADN condition, then $L_c$ is an isomorphism.
\end{corollary}

\begin{remark}\label{remHypSurfMin}
The tensor $g$ is in the kernel of $L_0$  if and only if 
$\partial M$ is a minimal hypersurface of $M$.
Indeed, the first three components of $L_0(h)$ vanish if $h=g$. For the fourth, note that if $g_t=(1+t)^2g$, then $H(g_t)=(1+t)^{-1}H(g)$, so by linearizing at $t=0$, we obtain
$$
2DH(g)g=-H(g).
$$
\end{remark}

\section{Case of Ricci Curvature}\label{sec:ppal}
In this section, we prove Theorem \ref{maintheorem}.
It is now well known that the Ricci equation is not elliptic due to the invariance
of curvature under diffeomorphism. We will modify this equation by drawing inspiration
from DeTurck's method. We thus add a gauge term so that this new equation becomes elliptic while ensuring
that its solutions remain solutions of the Ricci equation.
In order to construct our new equation, let us recall some differentials of operators.\\

We already have (see \cite{Besse}, for example)
$$
D\Ric(g)h=\frac12\Delta_Lh-\mathcal L_gB_g(h).
$$
The linearization in the first variable of the Bianchi operator is (see, for example, \cite{Delay:study})
$$
[D B_{(.)}(R)](g)h=-{R}B_g(h)+T(g,R)h,
$$
where ${R}$ is here identified with the corresponding endomorphism of $T^*M$  and
$$
[T(g,R)h]_j=T(g,R)^{kl}_j h_{kl}=\frac{1}{2}(\nabla^k R^l_j+\nabla^l
  R^k_j-\nabla_jR^{kl})
h_{kl}.
$$
We define 
$$
\Ric_\Lambda(g):=\Ric(g)+\Lambda g
$$
In particular, if $g$ is Einstein with $\Ric(g)=\lambda g$, we have 
$$
[D B_{(.)}(\Ric_\Lambda(g))](g)h=-(\lambda+\Lambda)B_g(h).
$$

Finally, recall that for any metric $g$, $B_{g}(\Ric_\Lambda(g))=0$
by the Bianchi identity.

The equation we choose to solve will be 
\bel{equaelliptique}
F(r,h)=(0,0,\gamma,\mathcal H),
\ee
with
\bel{defF}
F(h,r):=
\left\{
\begin{array}{ll}
\Ric_\Lambda(g+h)-R_\Lambda-{\mathcal L}_g((\lambda+\Lambda)^{-1}B_{g+h}(R_\Lambda)) &\text{in } M,\\
-(\lambda+\Lambda)^{-1}B_{g+h}(R_\Lambda) &\text{on }\partial M,\\
|(g+h)^T|^{\frac{-1}{n-1}}(g+h)^T &\text{on }\partial M,\\
H(g+h) &\text{on }\partial M,\\
\end{array}
\right.
\ee
where 
$$
R_\Lambda=\Ric(g)+\Lambda g+r=(\lambda+\Lambda) g+r,
$$
and $|(g+h)^T|$ denotes the determinant of $(g+h)^T$.\\

Let us first verify that the solutions of the new equation are solutions of
the equation we are interested in.

\begin{proposition}\label{sol}
Under the conditions of Theorem \ref{maintheorem}, if
  $h\in C^{k+2,\alpha}(M,\mathcal S_2)$ is small enough, and if the metric
$g+h$ is a solution of \eq{equaelliptique}, then it is a solution of 
$$
\left\{
\begin{array}{ll}
\Ric_\Lambda(g+h)=R_\Lambda &\text{in } M,\\
|(g+h)^T|^{\frac{-1}{n-1}}(g+h)^T = \gamma&\text{on }\partial M,\\
H(g+h)=\mathcal H &\text{on }\partial M,\\
\end{array}
\right.
$$
\end{proposition}

\begin{proof}
Applying $B_{g+h}$ to the first component of \eq{equaelliptique}, we note that $B_{g+h}[\Ric_\Lambda(g+h)]=0$ by the Bianchi identity. Thus, defining 
$$\omega:=(\lambda+\Lambda)^{-1}B_{g+h}(R_\Lambda),$$ we see that $\omega$ vanishes at the boundary and
$$
P_{g+h}\omega:=B_{g+h}[{\mathcal L}_g(\omega)]+(\lambda+\Lambda)\omega=0.
$$

The operator $P_g$ can be expressed in local coordinates as:
$$
(P_g\omega)_j=-\nabla^{i}\left[\frac{1}{2}(\nabla_i\omega_j+\nabla_j\omega_i)\right]
+\frac{1}{2}\nabla_j\nabla^i\omega_i+\Ric(g)_{j}^k\omega_k+\Lambda\omega_j.
$$
Commuting the derivatives and multiplying by $2$, we obtain 
$$2P_g=\Delta_g\omega+\Ric_g\omega+2\Lambda \omega=(\Delta_H+2\Lambda) \omega.$$
Since the operator $P_g$ with Dirichlet condition has a trivial kernel, by corollary \ref{DeltaHiso}, it is an isomorphism from
$C^{k+1,\alpha}_0(M,\mathcal T_1)$ to $C^{k-1,\alpha}(M,\mathcal T_1)$, then for small $h$ in $C^{k+2,\alpha}(M,\mathcal S_2)$,
$P_{g+h}$ with Dirichlet condition remains injective, we conclude that $\omega=0$.
\end{proof}

\begin{remark}
The fact that $B_{g+h}(R_\Lambda)$ vanishes proves that the identity map from $(M,g+h)$ to
$(M,R_\Lambda)$ is harmonic (see \cite{GL}, for example).
\end{remark}
We will now construct the solutions of \eq{equaelliptique} using an implicit function argument in Banach spaces.

\begin{proposition}
Under the conditions of Theorem \ref{maintheorem},
for any small $r\in C^{k+2,\alpha}(M,\mathcal S_2)$, for any conformal class $[\gamma]$ close to $[g^T]$ in $[C^{k+2,\alpha}(\partial M,{\mathcal S}_2^+)]$, and for any function $\mathcal H$ close to $H(g)$ in $C^{k+1,\alpha}(\partial M)$, there exists a unique 
$h$ close to zero in $ C^{k+2,\alpha}(M,\mathcal S_2)$ that solves \eq{equaelliptique}.
\end{proposition}

\begin{proof}
We consider $F$ as a mapping defined in a neighborhood of zero
in $C^{k+2,\alpha}(M,\mathcal S_2)\times C^{k+2,\alpha}(M,\mathcal S_2)$ with values in 
$C^{k,\alpha}(M,\mathcal S_2)\times C^{k+1,\alpha}(\partial M,\mathcal S_2)\times [C^{k+2,\alpha}(\partial M,{\mathcal S}_2^+)]\times C^{k+1,\alpha}(\partial M)$. We already have 
$$F(0,0)=(0,0,[g^T],H(g)).$$ Considering the differentials of the operators given at the beginning of the section, the differential of $F$ with respect to $h$ at $0$ is
$$
D_{h}F(0,0)=\frac12L_{2\Lambda},
$$
where $L_{2\Lambda}$ is described in Proposition \ref{LcFredholm}.
By Corollary \ref{DeltaLiso}, this operator is an isomorphism from $C^{k+2,\alpha}(M,\mathcal S_2)$ to $C^{k,\alpha}(M,\mathcal S_2)\times C^{k+1,\alpha}(\partial M,\mathcal S_2)\times C^{k+2,\alpha}(\partial M,\mathring {\mathcal S}_2)\times C^{k+1,\alpha}(\partial M)$.
The implicit function theorem allows us to conclude.

\end{proof}

\begin{remark}
If $L_{2\Lambda}$ has a kernel, we can still solve equation (\ref{mainequation}) but modulo a projection. In this case, we take
$$
R_\Lambda=\Ric_\Lambda(g)+r-\frac12\Pi(h)=(\lambda+\Lambda)g+r-\frac12\Pi(h),
$$
where $\Pi$ is the $L^2$-orthogonal projection onto the kernel; see \cite{Delay:RicciPara}.
According to Remark \ref{remHypSurfMin}, if $\Lambda=0$, the hypotheses of Theorem \ref{maintheorem} imply that the boundary of $M$ must not be a minimal hypersurface.
For the case $\Lambda=0$ with minimal boundary, if the kernel is one-dimensional, so generated by $g$, we can also solve (\ref{mainequation}) up to a multiplicative constant, as in \cite{Hamilton1984}.
Here, we can take (see \cite{Delanoe2003} and \cite{Delay:RicciPara}, Remark 4.4)
$$
R_0=e^{-\frac1{2n\lambda}\langle\Tr_gh\rangle}(\Ric(g)+r)=e^{-\frac1{2n\lambda}\langle\Tr_gh\rangle}(\lambda g+r)
$$
where
$$
\langle\Tr_gh\rangle=\frac1{\Vol_g(M)}\int_M\Tr_gh\,d\mu_g.
$$
\end{remark}

\section{Contravariant Ricci Operator}
We are interested here in the inversion of the (shifted) contravariant Ricci operator:
$$g\mapsto \overline\Ric_\Lambda(g):=\overline\Ric(g)+\Lambda g^{-1}$$
whose components in local coordinates are
$$\overline\Ric_\Lambda(g)^{ij}=g^{ik}g^{jl}\Ric(g)_{kl}+\Lambda g^{ij}=g^{ik}g^{jl}\Ric_\Lambda(g)_{kl}.$$
We will use the obvious notation
$$\overline\Ric_\Lambda(g)=g^{-1}\Ric_\Lambda(g)g^{-1}.$$
We adapt the steps of Section \ref{sec:ppal}. First, we have
$$
D\overline\Ric_\Lambda(g)h=g^{-1}[\frac12\Delta_Lh-\mathcal L_g B_g(h)-2\Ric h-\Lambda h]g^{-1}.
$$
Let us set $\overline B_g(\overline R)=B_{g}(g\overline Rg)$, so that if $\nabla\overline  R=0$, we obtain
$$
D\overline B_{(.)}(\overline R)h=-g\overline RB_g(h)+B_g(h\overline Rg+g\overline R h).
$$
If moreover $\overline R=(\lambda+\Lambda) g^{-1}$, we find
$$
D\overline B_{(.)}(\overline R)h=(\lambda+\Lambda) B_g(h)=g\overline RB_g(h).
$$
The gauge equation we choose to solve here will be
$$
F(h,\overline r)=(0,0,\gamma,\mathcal H)
$$
with
\bel{defFRicbar}
\overline F(h,\overline r):=
\left\{
\begin{array}{ll}
\overline F_1(h,r) &\mbox{in } M,\\
(\lambda+\Lambda)^{-1}\overline B_{g+h}(\overline R_\Lambda) &\mbox{on }\partial M,\\
|(g+h)^T|^{\frac{-1}{n-1}}(g+h)^T &\mbox{on }\partial M,\\
H(g+h) &\mbox{on }\partial M,\\
\end{array}
\right.
\ee
where
$$
\overline F_1(h,\overline r):=g[\overline\Ric_\Lambda(g+h)-\overline R_\Lambda]g+{\mathcal L}_g\{(\lambda+\Lambda)^{-1}\overline B_{g+h}[\overline R_\Lambda]\},
$$
and
$$
\overline R_\Lambda=\overline {\Ric_\Lambda}(g)+\overline r=(\lambda+\Lambda)g^{-1}+\overline r.
$$
\begin{theorem}\label{theoremRicContra}
Let $k\in\N\backslash\{0\}$ and $\alpha\in(0,1)$. Let $g$ be an Einstein metric with $\Ric(g)=\lambda g$ and let $\Lambda \in\R$ such that $\lambda+\Lambda\neq0$.
Assume that $4\lambda+2\Lambda$ is not in the spectrum of the Lichnerowicz Laplacian with ADN condition, nor in the spectrum of the Hodge Laplacian with Dirichlet condition.
Then for all $\overline r \in C^{k+2,\alpha}(M,\mathcal S^2)$ close to zero, for every conformal class $[\gamma]$ close to $[g^T]$ in $ [C^{k+2,\alpha}(\partial M,S_2^+)]$, and every function $\mathcal H$ close to $H(g)$ in  $C^{k+1,\alpha}(\partial M,\mathcal S_2)$, there exists a unique $h$ close to zero in $C^{k+2,\alpha}(M,\mathcal S_2)$ such that
$$
\left\{
\begin{array}{ll}
\overline \Ric_\Lambda(g+h)=\overline\Ric_\Lambda(g)+\overline r & \mbox{ on } M \\
 \,[ (g+h)^T]=[\gamma] &\mbox{ on } \partial M\\
H(g+h)=\mathcal H&\mbox{ on } \partial M
\end{array}
\right.
$$
Moreover, the map $(\overline r,[\gamma],\mathcal H)\mapsto h$ is smooth from a neighborhood of $(0,[g^T], H(g))$ into a neighborhood of zero between the corresponding Banach spaces.
\end{theorem}
\begin{proof}
We still have $\overline F_1(0,0)=0$ and since $g$ is Einstein,
$$
D_h\overline F_1(0,0)=\frac12\Delta_L-2\lambda-\Lambda.
$$
The corollary (\ref{DeltaLiso}) proves that the operator $D_h\overline F(0,0)$ is an isomorphism from $C^{k+2,\alpha}(M,\mathcal S_2)$ to $C^{k,\alpha}(M,\mathcal S_2)$.
By the implicit function theorem, for all $(\overline r,[\gamma],\mathcal H)$ close to $(0,[g^T],H(g))$ in their respective spaces, there exists $h$ close to zero in $C^{k+2,\alpha}(M,\mathcal S_2)$ such that
$$
\overline F(h,\overline r)=(0,0,\gamma,\mathcal H).
$$
We then apply $\overline B_{g+h}$ to the equation $\overline F_1(h,\overline r)=0$, obtaining
$$
\overline P_{g+h}\omega:=B_{g+h}\mathcal L_g\omega-(\lambda+\Lambda) \omega=0
$$
where
$$
\omega=(\lambda+\Lambda)^{-1}\overline B_{g+h}(\overline R_\Lambda)
$$
is zero at the boundary.
By hypothesis, the operator
$$
\overline P_g=\frac12(\Delta-\Ric_g)-(\lambda+\Lambda)=\frac12(\Delta_H-4\lambda-2\Lambda),
$$
with Dirichlet boundary condition
is injective (and is an isomorphism from
$C^{k+1,\alpha}_0(M,\mathcal T_1)$ to $C^{k-1,\alpha}(M,\mathcal T_1)$ by corollary \ref{DeltaHiso}), so if $h$ is small enough, $\overline P_{g+h}$ remains so, thus $\omega=0$.
\end{proof}

\section{Einstein-type Curvature}\label{secEin}
We show here that the method of Section \ref{sec:ppal} can also be adapted to other operators of Einstein type. For $\kappa$ and $\Lambda$ two real constants, we define the tensor
$$
\Ein(g):=\Ric(g)+\kappa R(g)g+\Lambda g.
$$
Thus, for example, when $\kappa=-\frac12$, we recover the Einstein tensor (with cosmological constant $\Lambda$), and if $\kappa=-\frac1{2(n-1)}$ and $\Lambda=0$, the Schouten tensor.

We study the inversion of the operator $\Ein$. Given a symmetric 2-tensor $E$, we seek $g$ such that
\bel{mainequationE}
\Ein(g)=E.
\ee
Since we have
$$
\Tr_g\Ein(g)=(1+n\kappa)R(g)+n\Lambda,
$$
the equation (\ref{mainequationE}) is equivalent to
$$
\Ric(g)=E-\frac{\kappa\Tr_g  E+\Lambda}{1+n\kappa}g.
$$
For any $E$, we define
$$
\mathcal B_g(E)=\div_gE+\frac{2\kappa+1}{2(1+\kappa n)}d\Tr_gE=B_g(E)-\frac{(n-2)\kappa}{2(1+\kappa n)}d\Tr_gE,
$$
such that the Bianchi identity translates here as
$$
\mathcal B_g(\Ein(g))=0.
$$
Knowing already the differential of $B_g(E)$ relative to the metric (see \cite{Delay:study}, for example), we find that the differential of this operator relative to the metric is
$$
D[\mathcal B_{(.)}(E)](g)h=-EB_g(h)+\frac{(n-2)\kappa}{2(1+\kappa n)}d\langle E,h\rangle
+T(E,h),
$$
where $E$ is here identified with the endomorphism of $T^*M$ corresponding to it, and
$$
T(E,h)_j=\frac12(\nabla_kE_{jl}+\nabla_lE_{kj}-\nabla_jE_{kl})h^{kl}.
$$
In particular, if $g$ is Einstein, with $\Ric(g)=\lambda g$, then $\Ein(g)=\tau g$ with $\tau=(1+n\kappa)\lambda+\Lambda$, and
$$
D[\mathcal B_{(.)}(E)](g)h=-\tau B_g(h)+\frac{(n-2)\kappa}{2(1+\kappa n)} \tau d\Tr_g h.
$$
By analogy with Section \ref{sec:ppal}, we define
$$
\mathcal F_1(h,e):=\Ric(g+h)-E+\frac{\kappa\Tr_{g+h}E+\Lambda}{1+\kappa n}{(g+h)}-\mathcal L_{g}(\tau^{-1}\mathcal B_{g+h}(E)),
$$
and
\bel{defFEin}
\mathcal F(h,r):=
\left\{
\begin{array}{ll}
\mathcal F_1(h,e) &\text{in } M,\\
-\tau^{-1}\mathcal B_{g+h}(E) &\text{on }\partial M,\\
|(g+h)^T|^{\frac{-1}{n-1}}(g+h)^T &\text{on }\partial M,\\
H(g+h) &\text{on }\partial M,\\
\end{array}
\right.
\ee
where $$E=\Ein(g)+e=\tau g+e.$$

We already have
$$
\mathcal F_1(0,0)=0.
$$
Here, we obtain
$$
D_h\mathcal F_1(0,0)h=$$
$$
\frac12\Delta_Lh+\frac{1}{1+\kappa n}
\left(\kappa\Tr_g\Ein(g) \;h+\Lambda h-\kappa\langle \Ein(g),h\rangle g\right)
$$
$$
-\frac{(n-2)\kappa}{2(1+\kappa n)}\mathcal L_{g}\tau^{-1}d\langle \Ein(g),h\rangle,
$$
thus,
$$
D_h\mathcal F_1(0,0)h=$$
$$
\frac12\Delta_Lh+\frac{1}{1+\kappa n}
\left(n\kappa\tau \;h+\Lambda h-\kappa\tau\Tr_gh\;g\right)
-\frac{(n-2)\kappa}{2(1+\kappa n)}\nabla\nabla\Tr_gh.
$$
This differential motivates us to define the operator
$$
\begin{array}{lll}
\mathcal P h&:=&\Delta_Lh+\frac{2(n\kappa\tau+\Lambda)}{1+kn}h
+\frac{\kappa}{n(1+\kappa n)}\Big({(n-2)}\Delta \Tr_gh-2 n\tau\Tr_gh\Big)\;g\\
&=&(\Delta_L+{2\kappa n\lambda+2\Lambda})h
+\frac{\kappa}{n(1+\kappa n)}\Big({(n-2)}\Delta \Tr_gh-2 n\tau\Tr_gh\Big)\;g.\\
\end{array}
$$
This respects the splitting $\mathcal S_2=\mathcal G\oplus\mathring {\mathcal S}_2$. In particular, if
$u$ is a function on $M$ and $\mathring h$ is a symmetric traceless 2-tensor field, we have
$$
\mathcal P(ug+\mathring h)=\frac{1}{1+\kappa n}p(u)g+\mathring P(\mathring h),
$$
where
$$
p(u)=(1+2(n-1)\kappa)\Delta u+2\Lambda u,
$$
and 
$$
\mathring P(\mathring h)=\left[\Delta_L+{2\kappa n\lambda +2\Lambda}\right]\mathring h.
$$

If $h=ug$, we find 
 $$D_h\mathcal F_1(0,0)(ug)=
\frac1{2}p(u) g-\frac{(n-2)n\kappa}{2(1+\kappa n)}\mathring\Hess \;u,
$$
where $\mathring\Hess \;u$ is the traceless part of the Hessian of $u$.
If $h=\mathring h$ is traceless, we find
 $$
D_h\mathcal F_1(0,0)(\mathring h)
=\frac12\mathring P(\mathring h).
$$

\begin{proposition}\label{isoEin} We assume $\kappa\neq-\frac1n,-\frac1{2(n-1)}$.
Let $$\upsilon=-2\kappa\left(\frac{2(n-1)\Lambda}{1+2(n-1)\kappa}+n\lambda\right).$$  If $-2\kappa n\lambda-2\Lambda$ is not in the spectrum of $\Delta_L+\upsilon\Tr(.)g$ with ADN$_\kappa$ conditions:
$$
\left\{
\begin{array}{ll}
B_g(h)-\frac{(n-2)\kappa}{2(1+\kappa n)} d\Tr_g h=0 &\text{on } \partial M\\
h^T-\frac{1}{n-1}\Tr_{g^T} (h^T) g^T=0&\text{on } \partial M\\
D H (g)h=0&\text{on } \partial M\\
\end{array}
\right.,
$$ then $D_h\mathcal F(0,0)$ is an isomorphism from $C^{k+2,\alpha}(M,\mathcal S_2)$ into $C^{k,\alpha}(M,\mathcal S_2)\times C^{k+1,\alpha}(\partial M,\mathcal S_2)\times C^{k+2,\alpha}(\partial M,\mathring {\mathcal S}_2)\times C^{k+1,\alpha}(\partial M)$.
\end{proposition}

\begin{proof}
It suffices to show that the map from $C^{k+2,\alpha}(M)\times C^{k+2,\alpha}(M,\mathring{\mathcal S}_2)$ into $C^{k,\alpha}(M)\times C^{k,\alpha}(M,\mathring{\mathcal S}_2)\times C^{k+1,\alpha}(\partial M,\mathcal S_2)\times C^{k+2,\alpha}(\partial M,\mathring {\mathcal S}_2)\times C^{k+1,\alpha}(\partial M)$ given by
$$(u,\mathring h)\longrightarrow
\left\{
\begin{array}{ll}
(p(u),\mathring P(\mathring h)) &\text{in } M\\
B_g(h)-\frac{(n-2)\kappa}{2(1+\kappa n)} d\Tr_g h &\text{on } \partial M\\
h^T-\frac{1}{n-1}\Tr_{g^T} (h^T) g^T &\text{on } \partial M\\
D H (g)h &\text{on } \partial M\\
\end{array}
\right.,
$$ where $h=ug+\mathring h$, is an isomorphism.

In order to do so, we will show that the map  
$$(u,\mathring h)\longrightarrow
\left\{
\begin{array}{ll}
\left(\frac1{1+2(n-1)\kappa}p(u),\mathring P(\mathring h)\right) &\mbox{ in } M\\
B_g(h)-\frac{(n-2)\kappa}{2(1+\kappa n)} d\Tr_g h &\mbox{on } \partial M\\
h^T-\frac{1}{n-1}\Tr_{g^T} (h^T) g^T &\mbox{on } \partial M\\
D H (g)h &\mbox{on } \partial M\\
\end{array}
\right.,
$$ 
is one. The operator  

$$h\longrightarrow
\left\{
\begin{array}{ll}
P(h):=\Delta_Lh+2\kappa n\lambda h+2\Lambda h+\upsilon\Tr(h)g &\mbox{ in } M\\
B_g(h)-\frac{(n-2)\kappa}{2(1+\kappa n)} d\Tr_g h &\mbox{on } \partial M\\
h^T-\frac{1}{n-1}\Tr_{g^T} (h^T) g^T &\mbox{on } \partial M\\
D H (g)h &\mbox{on } \partial M\\
\end{array}
\right.,
$$ 
is clearly a Fredholm operator of index zero, like $L_c$, by homotopy, it is therefore an isomorphism since it has a trivial kernel. Moreover, for any function $u$ and any traceless symmetric two-tensor $\mathring h$,
$$
P(u g+\mathring h)=\frac1{1+2(n-1)\kappa}p(u)g+\mathring P(\mathring h),
$$
which proves the stated result.
\end{proof}

\begin{theorem}\label{theoinvEin}
Let $k\in\mathbb{N}\backslash\{0\}$, $\alpha\in(0,1)$, $\kappa\neq -\frac1n,-\frac1{2(n-1)}$ and $\Lambda\in\mathbb{R}$. Let $g$ be an Einstein metric such that $\Ein(g)$ is non-degenerate. Under the assumptions of Proposition \ref{isoEin}, we further assume that $-2\kappa n\lambda-2\Lambda$ is not in the spectrum of $\Delta_H$ with Dirichlet condition. Then, for every small $ e\in C^{k+2,\alpha}(M,\mathcal S_2)$, for every conformal class $[\gamma]$ close to $[g^T]$ in $ [C^{k+2,\alpha}(\partial M,S_2^+)]$, and every function $\mathcal H$ close to $H(g)$ in $C^{k+1,\alpha}(\partial M,\mathcal S_2)$, there exists a unique $h$ close to zero in 
$C^{k+2,\alpha}(M,\mathcal S_2)$ such that
$$
\left\{
\begin{array}{ll}
\Ein(g+h)=\Ein(g)+e &\text{in } M,\\
|(g+h)^T|^{\frac{-1}{n-1}}(g+h)^T = \gamma&\text{on }\partial M,\\
H(g+h)=\mathcal H &\text{on }\partial M,\\
\end{array}
\right.
$$
Moreover, the map $(e,[\gamma],\mathcal H)\mapsto h$ is smooth in a neighborhood of $(0,[g^T], H(g))$ in a neighborhood of zero between the corresponding Banach spaces. 
\end{theorem}

\begin{proof}
From the hypotheses, $D_h\mathcal F(0,0)$ is an isomorphism.
The preceding calculations and the implicit function theorem then imply that for $( e,\gamma,\mathcal H)$ close to $(0,\gamma,\mathcal H)$ in their respective spaces, there exists $h$ close to zero in 
$C^{k+2,\alpha}(M,\mathcal S_2)$ such that
$$
\mathcal F(h,e)=(0,0,\gamma,\mathcal H).
$$
We now apply $B_{g+h}$ to the equation $ \mathcal F_1(h,e)=0$, thus
$$
B_{g+h}\mathcal F_1(h,e)=-\mathcal B_{g+h}(E)-B_{g+h}\mathcal L_{g}\tau^{-1}\mathcal B_{g+h}(E)=0.
$$
Setting $\omega=\tau^{-1}\mathcal B_{g+h}(E)$, we then have that $\omega$ vanishes at the boundary and
$$
P_{g+h}\omega :=B_{g+h}\mathcal L_{g}\omega+\tau\omega=0.
$$
But by hypothesis, the operator
$$
P_g=\frac12(\Delta-\Ric_g)+\tau=\frac12(\Delta +\Ric_g+2\kappa n\lambda+2\Lambda )=\frac12(\Delta_H +2\kappa n\lambda+2\Lambda )
$$
with Dirichlet boundary condition
is injective, so if $h$ is small, $P_{g+h}$ remains injective, hence $\omega=0$. 

\end{proof}

\section{Riemann-Christoffel Type Curvature}
We would like, just as in \cite{Delay:etude}, to show that the image of certain Riemann-Christoffel type operators are submanifolds in $C^\infty$, in the neighborhood of the metric $g$.
 We thus define:
$$
{\mathcal Ein}(g) = \Riem(g) + g {~\wedge \!\!\!\!\!\bigcirc ~} (a\Ric(g) + bR(g)g + cg),
$$
where ${~\wedge \!\!\!\!\!\bigcirc ~}$ is the Kulkarni-Nomizu product (\cite{Besse} p. 47), and
$$c = \frac{1+(n-2)a}{2(n-1)}\Lambda,\;\;b = \frac{\kappa[1+a(n-2)]-a}{2(n-1)},\;\;a \neq -\frac{1}{n-2}.$$ 

We then obtain:
$$
\Tr_g{\mathcal Ein}(g) = [a(n-2)+1]\Ein(g).
$$
We define the Riemann-Christoffel type version of ${\mathcal Ein}(g)$ by:
$$
[g^{-1}{\mathcal Ein}(g)]^i_{klm} := g^{ij}{\mathcal Ein}(g)_{jklm}.
$$

Consider ${\mathcal R}^1_3$, the subspace of ${\mathcal T}^1_3$ consisting of tensors satisfying:
$$
\tau^i_{ilm} = 0,\;\tau^i_{klm} = -\tau^i_{kml},\;
\tau^i_{klm} + \tau^i_{mkl} + \tau^i_{lmk} = 0.
$$
We define the Fréchet space:
$$C^{\infty} = \cap_{k \in \mathbb{N}}C^{k,\alpha},$$
endowed with the family of seminorms $\{\|.\|_{k,\alpha}\}_{k \in \mathbb{N}}$.

We then proceed similarly to \cite{Delay:etude} to prove that:

\begin{theorem}
Under the conditions of Theorem \ref{theoinvEin}, the image of the map:
$$
\begin{array}{lll}
C^{\infty}(M,\mathcal{S}_2) &\longrightarrow& C^{\infty}(M,\mathcal{R}_3^1)
\times [ C^{\infty}(\partial M,S_2^+)]\times C^{\infty}(\partial M,\mathcal{S}_2)\\
h &\mapsto &
\left\{
\begin{array}{ll}
(g+h)^{-1}{\mathcal Ein}(g+h)-(g)^{-1}{\mathcal Ein}(g) &\text{in } M,\\
|(g+h)^T|^{-\frac{1}{n-1}}(g+h)^T &\text{on }\partial M,\\
H(g+h)&\text{on }\partial M,\\
\end{array}
\right.
\end{array}
$$
is a smooth submanifold of  a neighborhood of $(0,[g^T],H(g))$.
\end{theorem}

\def\polhk#1{\setbox0=\hbox{#1}{\ooalign{\hidewidth
  \lower1.5ex\hbox{`}\hidewidth\crcr\unhbox0}}}
  \def\polhk#1{\setbox0=\hbox{#1}{\ooalign{\hidewidth
  \lower1.5ex\hbox{`}\hidewidth\crcr\unhbox0}}} \def\cprime{$'$}
  \def\cprime{$'$} \def\cprime{$'$} \def\cprime{$'$}
\providecommand{\bysame}{\leavevmode\hbox to3em{\hrulefill}\thinspace}
\providecommand{\MR}{\relax\ifhmode\unskip\space\fi MR }
\providecommand{\MRhref}[2]{%
  \href{http://www.ams.org/mathscinet-getitem?mr=#1}{#2}
}
\providecommand{\href}[2]{#2}

\end{document}